\documentclass[11pt,reqno]{amsart}
\usepackage{geometry}                
\usepackage{graphicx}
\usepackage{amssymb}
\usepackage{epstopdf}
\usepackage{mathtools}
\usepackage{array}
\usepackage{hyperref}
\usepackage{enumerate}
\usepackage{tikz}
\usepackage{tikz-cd}
\usepackage{mathtools}

\usepackage{amsthm}
\theoremstyle{definition}
\newtheorem{lemma}{Lemma}
\newtheorem{cor}[lemma]{Corollary}
\newtheorem{prop}[lemma]{Proposition}
\newtheorem{theorem}[lemma]{Theorem}
\newtheorem{notation}[lemma]{Notation}
\newtheorem{definition}[lemma]{Definition}
\newtheorem{example}[lemma]{Example}
\newtheorem{question}[lemma]{Question}

\newtheorem*{theoremsp1}{Theorem \ref{thm:pure}}
\newtheorem*{theoremsp2}{Theorem \ref{thm:main}}
\DeclareMathOperator{\Cir}{Circ}
\newcommand{\Circ}{\Cir^{*}}
\newcommand{\Circc}{\Cir^{**}}

\DeclarePairedDelimiter{\gen}{\langle}{\rangle}

\title{Maximizing Cliques in Shellable Clique Complexes}
\author{Corbin Groothuis}
\subjclass[2010]{Primary 05C35; Secondary 05E45}
\begin{document}
\begin{abstract}
In extremal graph theory, the problem of finding the elements of a given class of graphs which contain the most cliques traces its routes back to Tur\'an's famous theorem. We consider the implications of the connectivity property of simplicial complexes known as shellability on clique complexes associated with graphs. In this paper, we find the graphs which maximize cliques among all graphs on $n$ vertices with $\Delta(G)\leq r$ that have shellable clique complexes.
\end{abstract}
\maketitle

\section{Introduction and Background}
A clique $K_t$ is a subset of $t$ vertices in a graph such that every pair of vertices is joined by an edge. We use $k(G)$ to denote the number of cliques in a graph $G$ and $k_t(G)$ for the number of cliques of size $t$ in a graph $G$. Cliques are at the center of the following important question in extremal graph theory:

\begin{question}\label{qn:only}
Given a class of graphs $\mathcal{G}$ which graphs $G \in \mathcal{G}$ maximize $k(G)$?
\end{question}

This question of maximizing cliques in graphs dates back to Tur\'an, who found the maximal number of edges in a graph that doesn't contain a clique of size $r$. In fact, that family of graphs that maximizes the number of all cliques on $n$ vertices without a clique of size $r$, though it was Zykov \cite{zykov} who actually showed this. 
\begin{theorem}(Zykov)
If $G$ is a graph with $n$ vertices and $\omega(G)\leq r$, then
$$k(G)\leq k(T(n,r)),$$
where $T(n,r)$ is the $r$-partite Tur\'an graph on $n$ vertices. 
\end{theorem} 
Another parameter of interest is maximum degree. While a maximum degree condition imposes a restriction on clique number, $\Delta(T(n,r))$ grows with $n$. In \cite{cutler2014}, Cutler and Radcliffe answered Question \ref{qn:only} for the class of graphs with $\Delta(G)=r$.
\begin{theorem}(Cutler and Radcliffe)
Let $n=a(r+1)+b$ for $a \in \mathbb{N}, 0\leq b \leq r$. If $G$ is a graph on $n$ vertices with $\Delta(G)\leq r$, then
$$k(G)\leq k(aK_{r+1}\cup K_b).$$
\end{theorem} 
This question has been extended to the following question focusing on cliques of a given size. 
\begin{question}\label{qn:jk}
Given a class of graphs $\mathcal{G}$ which graphs $G \in \mathcal{G}$ maximize $k_t(G)$ for a given $t$?  
\end{question}
Extensions of this type have been shown for each of the classes considered earlier.
\begin{theorem}(Zykov)
If $G$ is a graph with $n$ vertices and $\omega(G)\leq r$, then 
$$k_t(G)\leq k_t(T(n,r)).$$
\end{theorem}
\begin{theorem}(Gan, Loh, and Sudakov \cite{gan})
Let $n=r+1+b$ for $0\leq b\leq r$. If $G$ is a graph on $n$ vertices with $\Delta(G)\leq r$ and $t\geq 3$, then $$k_t(G)\leq k_t(K_{r+1}\cup K_b).$$ 
\end{theorem}
This was later partially extended to graphs on any number of vertices by increasing the number of $K_{r+1}$'s.
\begin{theorem}(Cutler and Radcliffe \cite{cutler2017})
Let $r\leq 6$ and $n=a(r+1)+b$ for $a \in \mathbb{N},0\leq b\leq r$. If $G$ is a graph on $n$ vertices with $\Delta(G)\leq r$, then
$$k_t(G)\leq k_t(aK_{r+1}\cup K_b).$$
\end{theorem}

In \cite{kirsch}, Kirsch and Radcliffe also consider the class of graphs with $\Delta(G)\leq r$ and $\omega(G)\leq s$.

While these questions are about graphs, similar results have been framed in terms of simplicial complexes, such as the following corollary of the famed Kruskal-Katona theorem:

\begin{definition}
The colexigraphic order on $\mathcal{P}([n])$, $\leq_C$, is defined by $A\leq_C B$ if $\max(A\Delta B)\in B$. The colex graph $\mathcal{C}(n,m)$ is the graph with vertex set $[n]$ whose edges are the first $m$ edges in the colexigraphic order on $E(K_n)$. 
\end{definition}
Note that when $\binom{a}{2} < m < \binom{a+1}{2}$, this is just a $K_a$ where each vertex in the clique is then connected to the next vertex one at a time. 
\begin{theorem}(Corollary of Kruskal-Katona theorem \cite{kruskal} \cite{katona})
Let $G$ be a graph on $n$ vertices with $m$ edges. Then $k_t(G)\leq k_t(\mathcal{C}(n,m))$.
\end{theorem}

An extension of Kruskal-Katona by Frankl, F\"uredi, and Kalai \cite{frankl} also gives a bound on the number of cliques in graphs with a fixed number of edges and no $K_{r+1}$.

Therefore, it's natural to consider other properties coming from simplicial complexes by using the following complex related to cliques of a graph.
\begin{definition}
Given a graph $G$, let $K(G)$ be the complex given by $$K(G)=\{A\in V(G): G[A] \text{ is complete} \}.$$ This is called the \textit{clique complex}.
\end{definition}

These complexes are also commonly known as \textit{flag complexes}. In this language, maximizing cliques of a graph is equivalent to maximizing the number of faces in the clique complex of a graph. We will consider a property of complexes called shellability, a notion of connectivity which specifies the overlap larger faces must have with each other, and try to maximize the number of faces in a shellable clique complex with certain additional parameters. The shelling condition on isolated vertices is trivial, and the shellability requirement requires one or less non-trivial connected component, so we will consider only connected graphs throughout the paper.

With just shellability, Question \ref{qn:only} is uninteresting, as $K_n$ is trivially shellable and has the most cliques among graphs on $n$ vertices. The Tu\'ran graph is also shellable, so it is the optimal graph with shellable clique complex and $\omega(G)\leq r$. 

However for $n>r+1$, $aK_{r+1}\cup K_b$ has two non-trivial connected components, so it is not shellable. Therefore, the natural class of graphs to consider is graphs with maximum degree $r$ whose clique complexes are shellable.

We make the following notational definitions. 
\begin{definition}
Let $f_r(n)=\max_G\{k(G)\}$, for $G$ a graph on $n$ vertices with $\Delta\leq r$ and $K(G)$ shellable. 
\end{definition}
\begin{definition}
Let $f_{r,t}(n)=\max_G\{k_t(G)\}$, for $G$ a graph on $n$ vertices with $\Delta \leq r$ and $K(G)$ shellable.
\end{definition}
In addition, we will consider maximizing cliques over the family of \textit{pure} complexes (see Definition \ref{def:pure}). 
\begin{definition}
Let $f^*_{r,t}(n)=\max_G\{k_t(G)\}$, for $G$ a graph on $n$ vertices with $\Delta \leq r$ and $K(G)$ a pure shellable complex.
\end{definition}

The following observation motivates the form of our results.
\begin{lemma}\label{lem:folklore}(Folklore)
Let $G$ be a graph with $\Delta(G)\leq r$. Then $k_t(G)\leq \binom{r}{t-1} n$. In addition, $k(G)\leq 2^r n$.
\end{lemma}
\begin{proof}
For a vertex $v\in V(G)$, $k_{t-1}(G[N(v)])$ counts the number of $t$-cliques that $v$ is in. This is maximized when the neighborhood of $v$ is complete, i.e. when $N(v)=K_r$, and $k_{t-1}(K_r)=\binom{r}{t-1}$. Summing over all vertices gives an upper bound for $k_t(G)$.
\end{proof}
Since any sequence of graphs $G_1,G_2,\ldots,G_n$ with $|V(G_n)|=n$ has $k_t(G_n)\leq cn$ by the above lemma, we wish to determine for a class of graphs what the optimal $c$ is. Thus we will determine the values of $$\lim\limits_{n \to \infty}\dfrac{f_{r,t}(n)}{n} \text{ and }  \lim\limits_{n \to \infty}\dfrac{f_{r}(n)}{n}$$. 

We are interested in whether $\lim\limits_{n \to \infty}f_{r,t}(n)/n$ exists and if there is a family of graphs that achieve the limit for $t\leq r/2 + 1$. In doing so, we will also show that $\lim\limits_{n \to \infty}f_r(n)/n$ exists. While it is possible for any given $n$ and $r$ to find a graph that is extremal in $k(G)$, the overall improvement is negligible asymptotically and these extremal graphs are built by adding a handful of small facets to the families of graphs that we construct to achieve the limit.  

In this paper, we will answer Questions \ref{qn:only} and \ref{qn:jk} for the class of graphs with bounded maximal degree whose clique complexes are shellable. Section \ref{basicdef} introduces basic definitions and notation for clique complexes.
Due to an inherent parity influence on the problem, the even and odd cases are considered in separate sections. Section \ref{section:even} introduces the following class of clique complexes:
\begin{definition}
We will denote by $\Circ(n,r)$ the complex on $n$ vertices with max degree $r$ with $n- \lceil r/2 \rceil $ facets given by:
$$F_i=\{i,i+1,i+2,\ldots,i+\lfloor r/2 \rfloor \}, 1\leq i \leq n- \lceil r /2 \rceil $$
\end{definition}

The tight restrictions that maximum degree and shellability impose together allows us to reduce most of the problem to focusing on pure complexes of a certain size.
These complexes, based on circulant graphs turn out to be the only pure complexes with $\omega(G)=r/2+1$ for $n$ large enough.
\begin{theorem}\label{thm:pure}
Let $K(G)$ be a pure shellable complex with $\omega(G)=r/2+1$ and $\Delta(G)=r=2k$. If $n \geq r(r/4)(r/2+1)(r/2+2)+1$, then $K(G)=\Circ(n,r)$.
\end{theorem}
By leveraging Theorem \ref{thm:pure} with bounds on the number of smaller facets, we are able to prove our main result.
\begin{theorem}\label{thm:main}
Let $r=2k$ for some $k$. Then $\lim\limits_{n \to \infty} \dfrac{f_{r,t}(n)}{n}=\dbinom{r/2}{t-1}$.
\end{theorem}
A corollary gives the result for maximizing cliques in general.
\begin{theorem} \label{thm:mainz}
Let $r=2k$ for some $k$. Then $\lim\limits_{n \to \infty} \dfrac{f_r(n)}{n}=2^{r/2}$. In addition, for $G_n=\Circ(n,r)$, $\lim\limits_{n \to \infty} \dfrac{k(G_n)}{n}=2^{r/2}$, so the limit is achieved.
\end{theorem}

Lastly, in Section \ref{section:odd}, we consider the odd case and give some open questions.
\section{Basic definitions and Preliminaries}\label{basicdef}
We start by introducing some basic notation and definitions for simplicial complexes.
\begin{notation} Given a subset $\{F_1,F_2,\ldots,F_l\}$ of $2^{[n]}$, denote the simplicial complex generated by the subset by $\gen{F_1,\ldots,F_l}$.

\end{notation}
\begin{definition}
Given any simplicial complex $\Delta$, the f-vector corresponding to $\Delta$ is the vector $(f_0,\ldots,f_{d-1})$, where $f_i$ denotes the number of faces of dimension $i$ in $\Delta$.
\end{definition}
\begin{definition}\label{def:pure} A simplicial complex $\Delta$ is called \textit{pure of dimension $d-1$} if $|F|=d$ for all facets $F$ of $\Delta$. 
\end{definition}

There is a slight notational inconvenience between the algebraic view of measuring a simplicial complex by its dimension and the graph theoretic view of measuring the underlying graph by its clique number, as these numbers differ by one. In this paper, we will favor the latter view.
\begin{notation}
Given a simplicial complex $\Delta$, denote by $\omega(\Delta)$ the size of the largest facet of $\Delta$, that is, $\dim(\Delta)+1$.
\end{notation}

While we will mostly be considering non-pure complexes, the concept of a pure complex is useful for the following formulation of shellability.
\begin{definition}A simplicial complex $\Delta$ is called \textit{shellable} if there is an ordering of the facets $F_1,\ldots, F_s$ such that
$\gen{F_i}\cap \gen{F_1,\ldots,F_{i-1}}$ is pure of dimension $|F_i|-1$, for all $i$. Any such ordering of the facets is called a \textit{shelling order}.
\end{definition}
In some cases, it will be useful to talk about the underlying graph of a simplicial complex at each step of the shelling.
\begin{notation}
We will denote the underlying graph at step $i$ by $G_i$, where $$G_i=G\left\lbrack\bigcup_{j=1}^iF_j\right\rbrack.$$
\end{notation}
The facets in a shellable complex come in one of two flavors.
\begin{definition}
Given a shelling order $F_1,F_2,\ldots F_l$, $F_i$ is called a \textit{vertex adding facet} if $|V(G_i)|-|V(G_{i-1})|=1$. Otherwise $|V(G_i)|=|V(G_{i-1})|$ and $F_i$ is called a \textit{structural facet}.
\end{definition}
\begin{definition} Given a simplicial complex $\Delta$ with a given shelling order $\mathbf{F}$, let $r_i(\Delta, \mathbf{F})$ be the number of distinct vertices $v_k$ such that there exists $j<i$ such that $F_i-F_j=v_k$. We will use $r_i(\Delta)$ when the shelling order is clear from context. 
\end{definition}

\begin{example}
In the figure below we give a graph whose clique complex is denoted by $\Delta$ with a given shelling ordering $\mathbf{F}$. Note that this complex is pure.
\begin{center}
\includegraphics[]{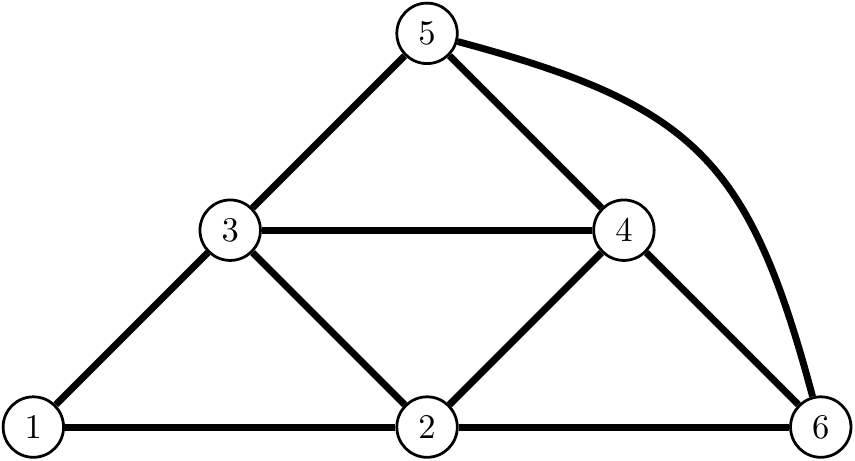}
\end{center}
$\mathbf{F}=(\{1,2,3\},\{2,3,4\},\{3,4,5\},\{2,4,6\},\{4,5,6\})$
\\$r_1(\Delta)=0,r_2(\Delta)=1,r_3(\Delta)=1,r_4(\Delta)=1,r_5(\Delta)=2$
\end{example}

While there can exist many different shelling orders for any given complex, the pure intersection requirement of shellability allows us to always ensure we can choose a shelling order with decreasing facet size.
\begin{lemma}[Rearrangement Lemma, Bj\"{o}rner \cite{bjorner}]
If $\Delta$ is a shellable complex, then there exists a shelling order for $\Delta$ such that $|F_1|,|F_2|,\ldots ,|F_s|$ is monotonically decreasing.
\end{lemma}
Unless otherwise stated, we will assume that our shelling orders are of this type. 

Due to the algebraic connection between simplicial complexes and Stanley-Reisner rings, some properties of Betti numbers can be used to find the sum of the $f$-vector by combining results of Herzog and Hibi \cite[Proposition 8.2.5]{herzog} and Sharifan and Varbaro \cite[Corollary 2.7]{sharifan}. We present a combinatorial argument for this result in the case where $\Delta$ is shellable.
\begin{prop}\label{prop:tformula}
Suppose that $\Delta$ is a shellable simplicial complex with shelling order $\mathbf{F}=F_1, F_2, \ldots F_\ell$ and $\omega(\Delta)=s$. Then with $r_i=r_i(\Delta,\textbf{F})$,

$$ f_{t+1}=\sum_{i=1}^\ell\binom{|F_i|-r_i}{t-r_i} .$$
\end{prop}
\begin{proof}
We will show this formula is correct by showing that when $F_i$ is added, it adds $\binom{|F_i|-r_i}{t-r_i}$ faces of size $t$. This formula is clear when $F_i$ is just an edge or a vertex, therefore we will consider only larger faces. 

Note that $\gen{F_i}\cap \gen{F_1,\ldots F_{i-1}}=\gen{H_1,H_2,\ldots H_m}$, where $|H_i|=|F_i|-1$. Since this is a decreasing shelling order, each $H_j$ is contained in a previous facet. Hence, $F_i\setminus H_j=v_j$ for some vertex $v_j \in F_i$. Then $r_i=m$ and let $R_i=\{F_i\setminus H_j: $ for $j\leq m$\}.

If $r_i=|F_i|$, then we are only adding the face $F_i$ itself, which matches our formula. If $r_i=0$, then there is no intersection with previous faces, so adding $F_i$ matches our formula. Now suppose $0<r_i<|F_i|$. 

When viewed as a face, $R_i$ is not in $\gen{F_1,\ldots F_{i-1}}$. If it were, then $R_i \subset H_j$ for some $j$. Then letting $v_j= F_i\setminus H_j$, $v_j$ must be in $R_i$, a contradiction. So $R_i$ is not a face. For any $A\subseteq F_i$ with $R_i\nsubseteq A$, $A \in \gen{F_1,\ldots F_{i-1}}$. To see this consider $v_j \in R_i\setminus A$ corresponding to $H_j$. Then $A \subseteq H_j$.

Therefore, when $F_i$ is added, every face of $F_i$ that doesn't contain $R_i$ is also added. There are $\binom{|F_i|-r_i}{t-r_i}$ of these. 
\end{proof}
Summing over all face sizes gives the following corollary.
\begin{cor}\label{prop:formula}
Suppose that $\Delta$ is a shellable simplicial complex with shelling order $F_1, F_2, \ldots F_l$ and $\omega(\Delta)=s$. Then with $r_i=r_i(\Delta)$, 

$$|\Delta|= \sum_{i=0}^s f_i=\sum_{i=1}^l 2^{|F_i|-r_i}.$$
\end{cor}

In a clique complex, structural facets have the same effect on the underlying graph, no matter what their size is.
\begin{lemma}\label{lem:structfacets}
Let $G$ be a graph such that $K(G)$ is a shellable clique complex and let $\textbf{F}$ be a given shelling order. If $F_i$ is a structural facet, then $|E(G_i)|-|E(G_{i-1})|=1$.
\end{lemma}
\begin{proof}
Suppose structural facet $F_i$ is added and this adds edges $xy$ and $uv$. Let $\gen{F_i}\cap \gen{F_1,\ldots F_{i-1}}=\gen{H_1,H_2,\ldots H_s}$, where $|H_i|=|F_i|-1$ as given by the definition of shellability. Then $\{x,y\}\cap H_i\leq 1$ and $\{u,v\}\cap H_i \leq 1$ for all $i$. 

If $|\{u,v,x,y\}|=3$, then without loss of generality $u=x$. As $F_i$ is not a vertex-adding facet, $x \in H_j$ for some $j$, but since $F_i$ adds edges $xy,xv, H_j\cap\{y,v\}=\emptyset$. Thus $|F_i\setminus H_i|\geq 2$, so $|H_i|<|F_i|-1$, a contradiction.

If $uv\cap xy = \emptyset$, then $|\{u,v,x,y\}\cap H_i|\leq 2$ for all $i$, so $|F_i\setminus H_i|\geq 2$, giving the same contradiction as above. So a structural facet must add exactly one edge.
\end{proof}
 
For graphs with pure complexes of dimension $m-1$ , it will be useful to define an auxiliary graph known as the $K(G,m)$ graph.
\begin{definition}
Let $G$ be a graph with pure clique complex $K(G)$ of dimension $m-1$. Then the facets of $K(G)$ are cliques of size $m$. We define the $K(G,m)$ graph to be the graph given by 
$$V(G)=\{F_i: i \text{ a facet}\}$$
$$E(G)=\{F_iF_j: |F_i\cap F_j|=m-1\}$$
\end{definition}

When the clique complex is shellable, this graph can be used to eliminate some possible shelling orders and determine ``$K_m$ connectivity'' of a graph, as we do in Section \ref{section:even}.
\section{The Even Case}\label{section:even}
Throughout the rest of this paper, we will consider a graph $G$ with $\Delta(G)\leq r$ and $K(G)$ shellable. In this section, we consider the case when $r$ is even. 

The natural intuition for constructing simplicial complexes with maximum possible number of faces is to include facets as large as allowed by the degree bound. If $n\leq r$, the best graph for maximizing cliques is $K_n$. It is trivially shellable and the maximum degree condition doesn't even come into play. For $n>r$, consider the graph consisting of a $K_{r+1}$ and $n-r-1$ isolated vertices. This gives a lower bound of $2^{r+1} + n-r$ on the maximum number of cliques. For certain $n$, this will be the extremal graph, but as $n$ increases, we are unable to build from this graph, as all of the vertices in our first component have maximum degree and no additional facets of size $2$ or more can be added. For $n>2^{r+2}/r$, any $r$-regular graph on $n$ vertices will have more edges (let alone faces) than this. Despite adding many faces at once, a $K_{r+1}$ simply limits our choices too much to be asymptotically desirable.

Let's step one dimension lower. If we start with a $K_{r}$, we do have the flexibility to add a second $K_{r}$, say without loss of generality $F_1=\{1,2,\ldots,r\}$ and $F_2=\{2,3,\ldots,r+1\}$. At this point, only vertices $1$ and $r+1$ have available degree and if they were connected, we would have instead started with a $K_{r+1}$. Since we can always take a shelling order to have decreasing facet sizes, at this point our only options for remaining facets are single edges and vertices. If the rest of the graph is expanded into a triangle-free $r$-regular graph, we would have slightly less than $2^r + 2^{r-1} + nr/2+ n$ cliques. It is clear that this eventually outpaces the $K_{r+1} \cup E_{n-r-1}$ graph, but any $r$-regular graph with a linear number of triangles in $n$ will eventually outpace this.

It seems that each dimension lower we go allows us to add more cliques of that size, but in exchange those facets contribute fewer cliques. The following proposition shows that in fact most clique counts are bounded independently of $n$. As we build our complex from the shelling order, maximum degree requirements on the underlying graph limit the total degree of vertices of $G[F_1,\ldots,F_i]$ to $|G[F_1,\ldots,F_i]|r$. The difference $$|G[F_1,\ldots,F_i]|r- \sum_{v \in G[F_1,\ldots,F_i]}d(v)$$ will be referred to as \textit{free degree}. If during the shelling process we run out of free degree, there is an obvious contradiction, as this would mean a vertex violates the maximum degree condition.   
\begin{prop}\label{prop:nottoobig}
For a graph $G$, with $\Delta(G)\leq r$ and $K(G)$ shellable, there are $O(1)$ (in terms of $n$) facets (hence faces) larger than $r/2 + 1$.
\end{prop}
\begin{proof}
Fix $r$ and let $\ell\leq r+1$ be the size of the largest facet. We start with a $K_\ell$ and $\ell$ vertices of degree $\ell-1$. Each vertex adding facet of size $m$ introduces a vertex of degree $m-1$ and increases the degree of $m-1$ other vertices by 1. We've therefore ``lost'' $m-1$ free degree and gained $r-(m-1)$ free degree.
\begin{align}
r-m+1-(m-1)=r+2-2m \tag{*}
\end{align}
which is negative when $r/2 +1 < m$. So each vertex adding facet larger than $r/2+1$ strictly decreases the amount of free degree left. Adding structural facets only compounds this problem, since it must add degree to at least two vertices and introduces no new ones. Since we start with a finite amount of free degree, adding facets larger than $r/2+1$ can only be done a fixed number of times dependent on $r$ but not $n$. 
\end{proof}
Because of Proposition \ref{prop:nottoobig}, the clique number of a graph can be misleading, as it doesn't capture the dimension of ``most'' of the graph as $n$ gets large. For convenience, we will let $m=r/2+1$ for the remainder of the paper. We show that in graphs with enough $K_{m}$'s, the dimension of $K(G)$ actually does correspond to what the dimension ``should'' be.

\begin{prop}\label{prop:toomanycooks}
Let $r=2k$. Suppose $K_m(G)=t$. If $t\geq t_0=t_0(r)$, then $\omega(G) = m$.
\end{prop}
\begin{proof}
From Lemma \ref{lem:folklore}, if $t\geq \binom{r}{m-1}\ell,$ then $n\geq \ell$, so we can choose $t$ large enough to have any fixed number of vertices we wish. From Proposition \ref{prop:nottoobig}, there is a bounded number of vertices in larger faces. Denote these vertices by $\bar{V}$. For $t$ large enough, there is an $F_i$, $|F_i|=m$, such that $F_i\cap \bar{V}=\emptyset$. Denote the first such facet in the shelling order by $\{v_1,v_2,\ldots, v_{r/2}\}$. This must connect to previous facets in an $r/2$ face, say there is a face $\{v',v_1,v_2,\ldots,v_{r/2 -1} \}$. But $v$ is in a larger facet, or else this would have been an earlier $m$-facet separate from the larger faces. So $d(v)>r/2 + r/2=r$, a contradiction. 
\end{proof}

In order to find a potential family of graphs with many $K_m$'s, we will turn to circulant graphs, which are a well-studied class of graphs with desirable symmetry properties.
\begin{definition}
A \textit{circulant graph} is a graph whose vertices are the elements of $\mathbb{Z}_n$ such that each vertex $i$ is adjacent to vertex $i+j$ and $i-j$ mod $n$ for each $j \in J$, for some $J \subseteq [\lfloor n/2 \rfloor]$. Denote by $\Cir(n,J)$ the circulant graph on $n$ vertices with associated set $J$. 
\end{definition}
In this paper we will consider only circulant graphs with $n/2 \notin J$, since then our graphs will be regular of degree $2|J|$. If $J=[t]$ for some $t\leq \lfloor n/2 \rfloor -1$, we form a $K_{t+1}$ centered at each vertex. This seems promising as letting $t=r/2$ gives an $r$-regular graph with many $K_{m}$ cliques. However, circulant graphs are not usually shellable.
\begin{example}
Given below is the graph of $\Cir(8,[3])$ (\cite{sagemath}).
\begin{center}
\includegraphics[width=2.0in, keepaspectratio=true]{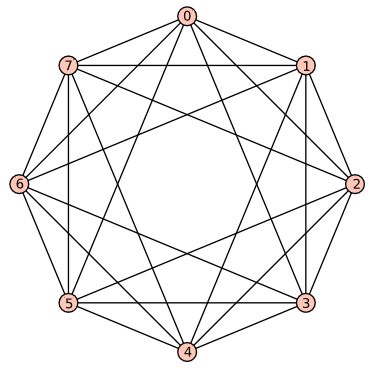}
\end{center}

The facets of the clique complex for this graph are $F_i=\{i,i+1 ,i+2,i+3\}$ for $i \in \mathbb{Z}_8$. But no shelling order exists, as the last facet in any ordering will intersect $\langle F_1,F_2,\ldots F_{l-1}\rangle$ in a disjoint vertex and edge, which is not pure.
\end{example}  
This is the only obstruction to shellability here however, so we will construct a family of ``close to circulant'' graphs where vertices are no longer viewed cyclically. We will define these in terms of the clique complex that they generate.
\begin{definition}
Let $G$ be the graph on $n$ vertices with $\Delta(G)\leq r$ defined by
$$ V(G)=[n]$$
$$ E(G)=\{x\sim y: |y-x|\leq \lfloor r/2 \rfloor\}$$

Denote by $\Circ(n,r)$ the clique complex with $G$ as an underlying graph.
\end{definition}
It is easy to see that the facets of $\Circ(n,r)$ are exactly 
$$F_i=\{i,i+1,i+2,\ldots,i+\lfloor r/2 \rfloor \}, 1\leq i \leq n- \lceil r /2 \rceil $$

If we restrict our attention to pure complexes without any structural facets, we can represent such complexes as a special rooted, labeled graph. 
\begin{definition}
Let $r=2k$. Given a graph $G$ with vertex set $[n]$ and $\omega(G)=m$ whose clique complex has no structural facets and without loss of generality $F_1\cap F_2=\{1,2,\ldots,r/2\}$, we'll construct a directed tree $T$ with root $v$ in the following way:
\begin{itemize}
\item Let $v=F_1 \cap F_2$.
\item $V(T)=\{F_i: i \in [l]\}\cup{v}$
\item $E(T)=\{v\rightarrow F_i: F_i \supset [r/2]\} \cup \{F_i \xrightarrow[]{x} F_j: i<j\text{ minimal such that}\\ (F_j\cap \langle F_1,F_2,\ldots,F_{j-1} \rangle )\subseteq F_i,[r/2]\subseteq F_i, \text{ where } x \in F_i\setminus F_j\}$ .
\end{itemize}
We call this a \textit{$K_m$ tree}. One can recover the original clique complex from the $K_m$ tree.
\end{definition}
Note that this is a tree, since $F_i\xrightarrow[]{x} F_j$ implies $i<j$ and the minimality condition gives that no vertex has more than one incoming edge. We will sometimes associate a vertex-adding facet with the vertex that it adds, so the facet that adds $j$ in the facet order will be denoted $F^*_j$. By convention, the tree vertex associated to vertex $i$ for $i\leq r/2$ in $G$ is $v$, and $F^*_i$ will refer to $v$ in those cases. Note that $F^*_i$ is not necessarily the $i$th facet in the facet ordering.

\begin{example}
As an example we have a clique complex $K(G)$ on the left with shelling order $(\{1,2,3\},\{1,2,4\},\{2,4,5\},\{1,4,6\},\{1,6,8\})$ and the associated $K_m$ tree on the right with the vertices labeled with both their associated facet and the associated $F^*$.
\begin{center}
\includegraphics[]{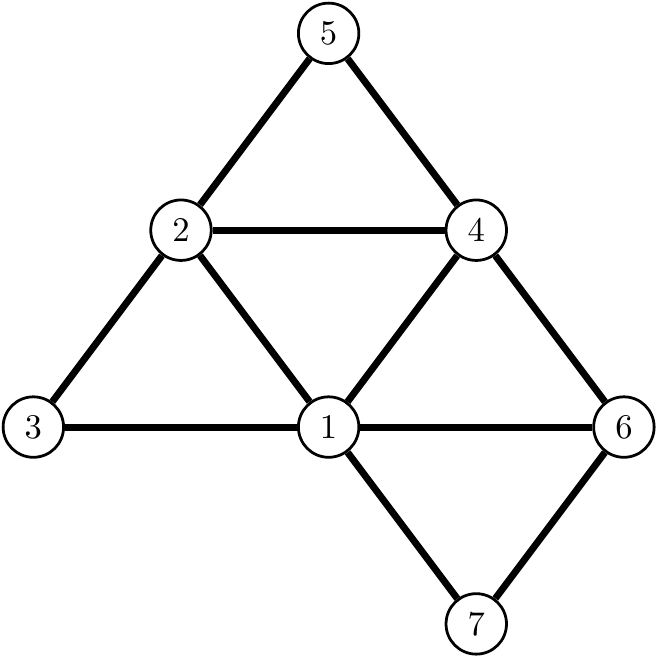}\qquad \qquad \includegraphics[scale=0.8]{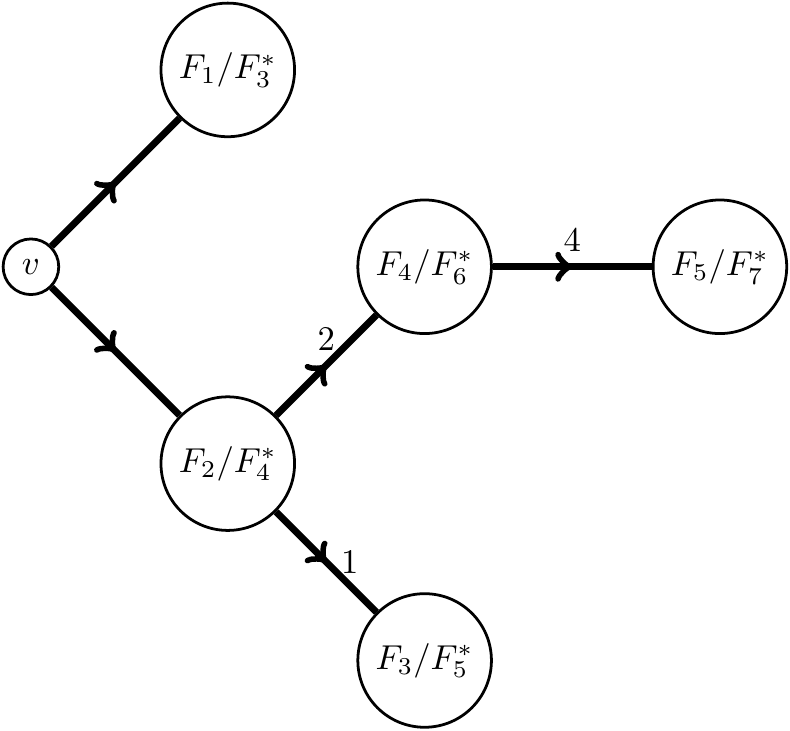}
\end{center}
\end{example}

A vertex $F^*_j$ is said to be \textit{reachable} from vertex $F^*_i$, if there is a directed path from $F^*_i$ to $F^*_j$ which does not use the label $i$. From this the degree of the vertices of $G$ can be recovered.

\begin{lemma}\label{lem:tree}
Let $T$ be the $K_m$ tree constructed from a graph $G$. Then for vertex $F^*_i$,

$\begin{cases}
d_G(i) = r/2 - 1 +|\{F^*_j: F^*_j\text{ is reachable from } F^*_i\}| &\text{if } i\leq r/2  \\
d_G(i) = r/2 + |\{F^*_j: F^*_j \text{ is reachable from } F^*_i\}| &\text{if } i > r/2.
\end{cases}$
\end{lemma}
\begin{proof}
In $G$, $i$ is only adjacent to any vertex it is in an $r/2+1$ clique with. Consider an $F^*_j$ reachable from $F^*_i$ and consider the labels $L(P)$ and associated $F^*$ vertices $V(P)$ of the directed path $P$ between them. Then $F^*_j=(F^*_i\cup V(P)) \setminus L(P)$ and $i\in F^*_j$. Any reachable $F^*_j$ corresponds to the vertex $j$ that it adds, so $i$ is in an $r/2+1$ clique with any $j$ corresponding to a reachable $F^*_j$. For $i\leq r/2$, $i$ is also adjacent to $j\in [r/2]\setminus\{i\}$ and for $i>r/2$, $i$ is adjacent to $j \in F^*_j\setminus\{i\}$.
\end{proof}
\begin{lemma}\label{lem:branch}
Any $K_m$ tree having $r/2+2$ non root vertices has $d(v)=2$ and the labels of edges are in bijection with $[r/2]$.
\end{lemma}
\begin{proof}
By definition $d(v)\geq 2$. Our tree has $r/2 + 2 - d(v)$ labeled edges. If the number of labels used is less than $r/2$ (say $r/2$ isn't used), then $d(r/2)=r/2 -1 + r/2 + 2= r+1$, a contradiction. Therefore, if $d(v)$ is greater than $2$ we have an immediate contradiction and if $d(v)=2$, then we must have a bijection between the labeled edges and $\{1,2,\ldots, r/2 \}$.
\end{proof}
This lemma is enough to show that these trees are very limited in how they can look.
\begin{theorem}\label{thm:unique}
Let $r=2k$. There is up to isomorphism as non-rooted trees only one $K_m$ tree with $r/2 + 3$ facets of size $m$ on $r+3$ vertices. 
\end{theorem}
\begin{proof}
By definition the tree corresponding to $\Circ(r+3,r)$ has the right number of facets and vertices. As a non-rooted tree this $K_m$ tree is merely a path. Suppose that $G\ncong \Circ(r+3,r)$. Construct the facet tree at some vertex $v$. Then in the tree, there must be a directed $K_{1,2}$ (besides the one at $v$), or else it is a path, a contradiction. Consider an $r/2 + 2$ subtree containing this $K_{1,2}$. Let $v_1$ and $v_2$ be the two neighbors of $v$. By Lemma \ref{lem:branch}, these are the only two neighbors of $v$.

Case 1: $v_1$ and $v_2$ both have descendants. Then our tree has at least $3$ labeled leaves, $F^*_{\ell_1},F^*_{\ell_2},F^*_{\ell_3}$. Each leaf $F^*_{\ell_i}$ corresponds to a vertex in $[r/2]$ of full degree, as the vertex whose label is on the edge to $F^*_{\ell_i}$ is reachable by every vertex but $\ell_i$. To go from the subtree to the full tree, we add the last edge and vertex $w=F*_j$ in. The added vertex has tree depth at most $r/2$. But the path from $v$ to $w$ must have the labels of all leaves or at least one of them, say $\ell_1$ will have degree $r+1$, as then $j$ will be reachable by $\ell_1$. But the path can contain at most $2$ of those labels, since these labeled leaves before, hence were not in a path together, and we have added only one edge. 

Case 2: Without loss of generality $v_2$ has no descendants. Then if $v_1=F^*_j$, $j$ has degree $r/2 + 1$ since $j\notin [r/2]$ and the labels of the subtree edges are in bijection by Lemma \ref{lem:branch}. But then going from the subtree to the full tree, the added edge must have label $j$ but also have the labels of the leaves.

Hence there cannot be two $K_{1,2}$, so the $K_m$ tree as a non-rooted tree is isomorphic to a path. 
\end{proof}
\begin{cor}
Let $r=2k$. There is up to isomorphism only one pure complex $\Delta=K(G)$ with $V(G)=n$, $\Delta(G)\leq r$, and $K(G)$ shellable.
\end{cor}

Hence, $\Circ(r+3,r)$ is the only possibility.
The presence of a $\Circ(r+3,r)$ subcomplex in a pure complex with $\omega(G)=m$ is actually enough to identify what complex we are working with.
\begin{prop}\label{prop:tendril}
Let $r=2k$. If $K(G)$ is a pure shellable complex with $\omega(G)=m$ and $K(G)$ contains a $\Circ(r+3,r)$, then $K(G)=\Circ(n,r)$.
\end{prop}

\begin{proof}
Let $t$ be the largest integer such that $\Circ(t,r)$ is a subcomplex of $K(G)$. In particular, $m\geq r+3$. This portion of the graph has $m-r$ vertices with degree $r$ in the middle and $r/2$ vertices on either end. Consider how this subcomplex connects to the rest of our complex. Since the middle vertices have maximum allowable degree, the only way to connect in an $r/2$ face is to take the first $r/2$ vertices or the last $r/2$ vertices in the subcomplex, as there is no $r/2$ clique with vertices from both ends. But then adding that facet forms a $\Circ(t+1,r)$, so $t=n$.
\end{proof}

\begin{theoremsp1}
Let $G$ be a graph with $\omega(G)=m$ and $\Delta(G)=r=2k$, where $K(G)$ is a pure shellable complex. If $n \geq r(r/4)(r/2+1)(r/2+2)+1$, then $K(G)=\Circ(n,r)$.
\end{theoremsp1}
\begin{proof}
Consider the $K(G,m)$ graph. For any $w \in K(G,m), d(w)\leq r$. To see this, without loss of generality, let $w=[r/2+1]$. Consider the $r/2$ faces containing $1$ in $[r/2 + 1]$. There are at most $r/2$ other facets containing at least one of these, or else $d_G(1)>r$. If you consider the remaining $r/2$ face that doesn't contain $1$, there are at most $r/2$ other facets containing that face or else $d_G(2)>r$. In total, $d(w) \leq r/2+r/2=r$. 

By Proposition \ref{prop:nottoobig}, $K(G)$ has at most $(r/2)(r/2+1)/2$ structural facets. If the structural facets are removed from $K(G,m)$, there are at most $r(r/2)(r/2+1)/2$ disconnected components as each removed vertex has at most $r$ neighbors. Suppose $n \geq r(r/4)(r/2+1)(r/2+2)+1$. Then by pigeonhole, one of these components has at least $r/2+3$ vertices. Take the connected, induced subgraph of $K(G,m)$ on $r/2+3$ of these vertices. The corresponding facets in $K(G,m)$ are all vertex-adding facets and are connected, and thus these facets constitute $r+3$ vertices in $G$. But then this is $r/2 + 3$ $m$-facets on $r+3$ vertices, so by Theorem \ref{thm:unique}, they make up a $\Circ(r+3,r)$. But then $K(G)$ contains a $\Circ(r+3,r)$, so by Proposition \ref{prop:tendril}, $K(G)=\Circ(n,r)$.
\end{proof}
So $\Circ(n,r)$ is optimal among pure complexes with $\omega(G)=m$ because it is actually the only game in town.

 Let's use Corollary \ref{prop:formula} to calculate $k_t(\Circ(n,r))$. Since we have $r_1(\Circ(n,r))=0$ and $r_i(\Circ(n,r))=1$ for $i>1$, for $t\leq m$
\begin{align*}
k_t(\Circ(n,r))&=\binom{m}{t} + (n-m)\binom{m-1}{t-1}\\
&=\binom{m-1}{t-1}\left(\dfrac{m}{t}+n-m\right).
\end{align*}

From Proposition \ref{prop:nottoobig}, for $n$ large enough we know that a pure complex must have $\omega(G)\leq r/2+1$. We will show that no pure complex does better than $\Circ(n,r)$.

Suppose $K(G)$ is a pure complex with $\omega(G)=s=m-i$. There are $n-m+i+1$ vertex adding facets and suppose our complex has $w$ additional structural facets. 

Like above, besides the first, each of the vertex adding facets has at least one distinct vertex from previous facets. Each of the $w$ additional structural facets has at least two distinct vertices from previous facets (it's easy to see that if it had one it would be vertex adding). So for $t\leq m$,
\begin{equation}\label{eqn:needlabel}
k_t(G)\leq \binom{m-i}{t} + (n-m+i)\binom{m-i-1}{t-1} + w\binom{m-i-2}{t-2}.
\end{equation}
\begin{lemma}\label{lem:epsilon}
With $K(G)$ as described above, $w<(i+o(1))n$ for $n$ large enough.
\end{lemma}
\begin{proof}
Using formula $(*)$ from Proposition \ref{prop:nottoobig}, we start with $s(s-1)$ free degree and each other vertex adding facet adds $(r-2s+2)$ additional free degree. From Lemma \ref{lem:structfacets}, each structural facet adds exactly one edge, so each reduces free degree by $2$. Therefore,
\begin{align*}
(s)(s-1)+(n-s)(r-2s+2)-2w&\geq 0\\
1/2(s)(s-1)+ 1/2(n-s)(r-2s+2)&\geq w\\
(r/2+1-s)n+o(1)n &> w\\
(i+o(1))n &>w
\end{align*}
\end{proof}
We will require the use of the following lemma about binomial coefficients.
\begin{lemma}\label{lem:binom}
Let $a,b,c \in \mathbb{N}$ such that $0\leq c \leq b \leq a$. Then 
$$ \binom{b}{c}+(a-b)\binom{b-1}{c-1} \leq \binom{a}{d} \text{ for } c\leq d \leq c+(a-b)-1 $$
\end{lemma}
\begin{proof}
As a function of $k$, $\dbinom{n}{k}$ is concave down, so the minimum value of $\dbinom{a}{d}$ is achieved at either $d=c$ or $d=c+(a-b)+1$. We will show that the left-hand side is less than both of these. Note that when $(a-b)=0$, the identity trivially holds, so without loss of generality, $(a-b)\geq 1$. 

Consider a set of $a$ points separated into a cluster $C_1$ of $b$ points, one of which is marked point $x$, and a cluster $C_2$ of $a-b$ points. The left hand side counts the number of ways to pick $c$ points from $C_1$ plus the number of ways to pick 1 point from $C_2$ and $c-1$ points from $C_1\setminus\{x\}$. 

These are all examples of choices of $c$ points from $a$ total points, so this expression is less than or equal to  $\dbinom{a}{c}$. In order to show that it is also less than or equal to $\dbinom{a}{c+(a-b)-1}$, we will exhibit an injection.

Mark a point $y$ in $C_2$. We will map a choice of $c$ points $D\in C_1$ to $D \cup (C_2\setminus\{y\})$. A choice of $c-1$ points $D\in (C_1\setminus \{x\})$ and $w \in  (C_2\setminus\{y\})$ will be mapped to $D\cup\{x\}\cup (C_2\setminus\{w\})$. Lastly, a choice of $c-1$ points $D\in (C_1 \setminus \{x\})$ and $y$ will be mapped to $D\cup C_2$.

The combination of the three choices of $c$ points again is counted by $\dbinom{b}{c} + (a-b)\dbinom{b-1}{c-1}$. Note that the images of these choices are disjoint and each one consists of $c+(a-b)-1$ points, so this count is less than or equal to $\dbinom{a}{c+(a-b)-1}$.
\end{proof}

We will now show that asymptotically $\Circ(n,r)$ maximizes cliques of size $t$ among all graphs with $\Delta(G)\leq r$ with pure shellable complexes.
\begin{prop}
Let $r=2k$. Then for $t\leq m$, $\lim\limits_{n \to \infty} \dfrac{f^*_{r,t}(n)}{n}= \dbinom{m-1}{t-1}$
\end{prop}
\begin{proof}
For $n$ large enough, Theorem \ref{thm:pure} says we need only look at $\omega(G)< m$. 

Fix a graph $G$ with $n$ vertices. With notation as above and given that $i>1$ we can use (\ref{eqn:needlabel}) to get:
\begin{align*}
k_t(G)&\leq \binom{m-i}{t} + (n-m+i)\binom{m-i-1}{t-1}+w\binom{m-i-2}{t-2}\\
&\leq n\biggl(\binom{m-i-1}{t-1}+(i+o(1))\binom{m-i-2}{t-2}\biggr)+o(r)\\
&< n\binom{m-1}{t-1} + o(r).
\end{align*}
The last inequality comes from applying Lemma \ref{lem:binom} with $a=m-1, b=s-j-1,$ and $c=t-j-1$. Taking limits gives $\lim\limits_{n \to \infty} \dfrac{f^*_r(n)}{n}\leq \binom{m-1}{t-1}$. Since $k(\Circ(n,r))=\binom{m-1}{t-1}n + o(n)$, equality is achieved.
\end{proof}

Now that we have a handle on graphs with small clique number, we can combine the restrictions on larger cliques to make a statement about all shellable complexes with $\Delta(G)\leq r$. 
\begin{theoremsp2}
Let $r=2k$ for some $k \in \mathbb{Z}$. Then $$\lim\limits_{n \to \infty} \dfrac{f_{r,t}(n)}{n}=\dbinom{m-1}{t-1}$$.
\end{theoremsp2}
\begin{proof}

Let $q$ be the number of vertices of $G$ contained in $K_m$'s and $t_0=t_0(r)$ be the parameter specified by Proposition \ref{prop:toomanycooks}. If $q < \max\{t_0,r(r/4)(r/2+1)(r/2+2)+1\}+r$, then $k_m(G)$ is bounded in terms of $r$. By this fact and Proposition \ref{prop:nottoobig}, the number of vertices $y_0$ contained in cliques larger than size $m$ or in cliques with vertices contained in cliques larger than size $m$ is bounded in terms of $r$. We can give an upper bound for the number of $t$-cliques contained in those vertices by $\binom{y_0}{t}$. Now that we are no longer restricting to just pure complexes, we cannot refer to structural facets or vertex-adding facets without specifying size.  Denote by $b_j$ the number of vertex-adding facets of size $j$ and by $c_j$ the number of structural facets of size $j$ not contained in the $y_0$ vertices from before. Let $s$ be the size of the largest facet less than $m$ and $i=m-s$.

Then we can write
$$k_t(G)\leq \binom{y_0}{t}+\sum_{j=0}^s \biggl(b_{s-j}\binom{m-i-j-1}{t-j-1} + c_{s-j}\binom{m-i-j-2}{t-j-2}\biggr)$$
By applying Lemma \ref{lem:epsilon} to the skeleta at each dimension, we have that
\begin{align*}
c_s&<(i+o(1))b_s \\
c_s+c_{s-1}&<(i+o(1))b_s +(i+1+o(1))b_{s-1}\\
\vdots \\
\sum_{j=0}^s c_{s-j} &< \sum_{j=0}^s (i+j+o(1))b_{s-j}
\end{align*}
Multiplying the $j$th equation by $\dbinom{s-3-j}{t-2-j}$ and summing together gives

\begin{align}\sum_{j=0}^s \binom{s-j-2}{t-j-2}c_{s-j}< \sum_{j=0}^s \binom{s-j-2}{t-j-2}(i+j+o(1))b_{s-j} \tag{**}
\end{align}
Then by applying $(**)$, given $m>=2$, 
 \begin{align*}
k_t(G)&\leq \binom{y_0}{t}+ \binom{m-i}{t}+\sum_{j=0}^s \biggl(b_{s-j}\binom{m-i-j-1}{t-j-1} + c_{s-j}\binom{m-i-j-2}{t-j-2}\biggr)\\
&< \binom{y_0}{t} + \binom{m-i}{t}+\sum_{j=0}^s b_{s-j}\biggl(\binom{s-j-1}{t-j-1}+(i+j+o(1))\binom{s-j-2}{t-j-2}\biggr)\\
&< \binom{y_0}{t}+\binom{m-i}{t} + \max_j\left\{\binom{s-j-1}{t-j-1}+\binom{s-j-2}{t-j-2}(m-s+j)\right\}\sum_{j=0}^s b_j\\
&\leq \binom{y_0}{t} + \binom{m-i}{t} + \binom{m-1}{t-1}\sum_{j=0}^s b_j \tag{*}\\
&\leq \binom{y_0}{t} + \binom{m-i}{t} + \binom{m-1}{t-1}(n-s+1)\\
&\leq \binom{m-1}{t-1}n+o(n)
\end{align*}
 where (*) comes from applying Lemma \ref{lem:binom} with $a=m-1,b=s-j-1,$ and $c=t-j-1$.
 
 If $q\geq \max\{t_0,r(r/4)(r/2+1)(r/2+2)+1\}+r$, in particular $k_m(G)>t_0$, so by Proposition \ref{prop:toomanycooks}, $\omega(G)\leq r/2+1$. Additionally, $q\geq r(r/4)(r/2+1)(r/2+2)+1$, so if $F_1,\ldots,F_i$ are the facets of size $m$, then $\gen{F_1\ldots,F_i}$ is pure with clique number $m$ so by Theorem \ref{thm:pure} it is $\Circ(q,r)$. Lower sized structural facets can only be built on the $r/2$ vertices on each end of $\Circ(q,r)$ and any additional vertices that come from vertex-adding facets of lower faces. We can bound the structural facets on those $r$ total vertice on the end by $2^r$. Then using the same process as above we have
\begin{align*}
 k_t(G)\leq& \binom{r}{t} + \binom{m}{t} + (q-m)\binom{m-1}{t-1} \\ +&\sum_{j=0}^{m-1} b_{m-1-j}\biggl(\binom{m-1-j-1}{t-j-1} + c_{m-1-j}\binom{m-1-j-2}{t-j-2}\biggr)\\
 \leq& \binom{r}{t} + \binom{m}{t} + (q-m)\binom{m-1}{t-1}+ \binom{m-1}{t-1}(n-q-m)\\
\leq& \binom{m-1}{t-1}n + o(n)
\end{align*}
In both cases, taking limits gives that $\lim\limits_{n \to \infty} \dfrac{f_{r,t}(n)}{n}\leq \binom{m-1}{t-1}$. As before, $\Circ(n,r)$ achieves equality.   
\end{proof}
As a corollary, we get the asymptotic value for $f_r$ as well.
\begin{cor}
Let $r=2k$ for some $k \in \mathbb{Z}$. Then $\lim\limits_{n \to \infty} \dfrac{f_r(n)}{n}=2^{r/2}$. 
\end{cor}
\begin{proof}
From Theorem \ref{thm:main}, we have $k_t(G)$ for $t\leq m$, and by Proposition \ref{prop:nottoobig}, $k_t(G)=o(n)$ for $t>m$.
\end{proof}
\section{The Odd Case}\label{section:odd}
When $r$ is odd, less is known. Regular circulant graphs don't exist for $r$ odd, and our construction leaves room for an improvement to the linear term. 
\begin{definition}
For $r$ odd and $n>r$, define $\Circc(n,r)$ as

$$\Circc(n,r)=\Circ(n,r)\cup \{\{i,i+\lceil n/2 \rceil \}: i\leq \lfloor n/2 \rfloor \}$$  
\end{definition}
Let $m=\lfloor r/2 \rfloor + 1$ for the remainder of this section. We don't have the exact formulation desired for an analogue of Proposition \ref{prop:toomanycooks}, but we do get the following.
\begin{prop}\label{prop:oddtoomanycooks}
Let $r=2k+1$. Suppose $G$ has $t$ $K_{m}$'s. Then for $t\geq t_0=t_0(r)$, $\omega(G) = \lceil r/2\rceil  + 1 $.
\end{prop}
The proof is identical, but since we have the wiggle room for one degree, we can have a slightly larger clique number.

As before, consider a pure complex with $\omega(G)=s=m-i$ 
with $n-s+1$ vertex-adding facets and $w$ structural facets. We have an analogue of Lemma \ref{lem:epsilon}:
\begin{lemma}\label{lem:oddepsilon}
 With $K(G)$ described above, $w<(i+1/2+o(1))n$ for $n$ large enough.
\end{lemma}
\begin{proof}
Using formula $(*)$ from Proposition \ref{prop:nottoobig}, we start with $(s)(s-1)$ free degree and for each other vertex adding facet we get $(r-2s+2)$ additional free degree. From Lemma \ref{lem:structfacets}, we have that each structural facet adds exactly one edge, so each reduces free degree by $2$. Therefore,
\begin{align*}
s(s-1)+(n-s)(r-2s+2)-2w&\geq 0\\
\frac{s}{2}(s-1)+ 1/2(n-s)(r-2s+2)&\geq w\\
(r/2+1-s)n+o(1)n &> w\\
(m+1/2-s+o(1))n &>w\\
(i+1/2+o(1))n &>w.\\
\end{align*}
\end{proof}
There is not a clear analogue of Theorem \ref{thm:unique}, as the additional free degree allows for a number of non-isomorphic trees. While it is possible to show that $\Circc(n,r)$ is asymptotically optimal compared to graphs with $\omega(G)<\lfloor r/2 \rfloor + 1$, there is possible room for improvement at $\omega(G)=\lfloor r/2 \rfloor + 1$. That leads to two questions.
\begin{question}
Let $r=2k+1$ for some $k \in \mathbb{Z}$. Is there a pure complex $K(G)$ with $\omega(G)=\lfloor r/2 \rfloor + 1$ with $k(G)\geq k(\Circ(n,r))$ for large $n$?
\end{question}
In addition to considering the total number of cliques, the additional flexibility with odd degree may allow for maximizing $k_t$ for different graphs.
\begin{question}
What is $\lim\limits_{n \to \infty}\dfrac{f_{r,t}(n)}{n}$ for $r$ odd? Is there a family of graphs that simultaneously achieve this limit for $t\leq \lfloor r/2 \rfloor + 1$?
\end{question}
\bibliographystyle{plain}

\bibliography{myreferences}

\begin{thebibliography}{10}

\bibitem{bjorner}
Anders Bj{\"o}rner and Michelle Wachs.
\newblock Shellable nonpure complexes and posets. i.
\newblock {\em Transactions of the American mathematical society},
  348(4):1299--1327, 1996.

\bibitem{cutler2014}
Jonathan Cutler and A.J. Radcliffe.
\newblock The maximum number of complete subgraphs in a graph with given
  maximum degree.
\newblock {\em Journal of Combinatorial Theory, Series B}, 104:60--71, 2014.

\bibitem{cutler2017}
Jonathan Cutler and A.J. Radcliffe.
\newblock The maximum number of complete subgraphs of fixed size in a graph
  with given maximum degree.
\newblock {\em Journal of Graph Theory}, 84(2):134--145, 2017.

\bibitem{sagemath}
The~Sage Developers.
\newblock {\em SageMath, the Sage Mathematics Software System (Version 8.1)}.
\newblock http://www.sagemath.org.

\bibitem{frankl}
Peter Frankl, Zolt{\'a}n F{\"u}redi, and Gil Kalai.
\newblock Shadows of colored complexes.
\newblock {\em Mathematica Scandinavica}, pages 169--178, 1988.

\bibitem{gan}
Wenying Gan, Po-Shen Loh, and Benny Sudakov.
\newblock Maximizing the number of independent sets of a fixed size.
\newblock {\em Combinatorics, Probability and Computing}, 24(3):521--527, 2015.

\bibitem{herzog}
J{\"u}rgen Herzog and Takayuki Hibi.
\newblock Monomial ideals.
\newblock {\em Monomial Ideals}, pages 3--22, 2011.

\bibitem{katona}
Gyula Katona.
\newblock A theorem of finite sets.
\newblock In {\em Classic Papers in Combinatorics}, pages 381--401. Springer,
  2009.

\bibitem{kirsch}
Rachel Kirsch and A.J. Radcliffe.
\newblock Many $k_t$'s; no large cliques or stars.
\newblock {\em arXiv:1712.07769}, 2017.

\bibitem{kruskal}
Joseph~B Kruskal.
\newblock The number of simplices in a complex.
\newblock {\em Mathematical optimization techniques}, 10:251--278, 1963.

\bibitem{sharifan}
Leila Sharifan and Matteo Varbaro.
\newblock Graded betti numbers of ideals with linear quotient.
\newblock {\em Le Matematiche}, 63(2):257--265, 2009.

\bibitem{zykov}
A.A. Zykov.
\newblock {\em On Some Properties of Linear Complexes}.
\newblock American Mathematical Society translations. American Mathematical
  Society, 1952.

\end{thebibliography}
\end{document}